\documentclass[11pt, oneside]{amsart}   	
\usepackage{amsmath,amsthm,amscd,amsfonts, amssymb, mathrsfs}

\usepackage{geometry}                		
\usepackage{graphicx}				

\makeatletter
\@namedef{subjclassname@2020}{\textup{2020} Mathematics Subject Classification}
\makeatother

\newcommand{\sect}[1]{\section{#1}\setcounter{equation}{0}}
\newcommand{\subsect}[1]{\subsection{#1}}

\font\mbn=msbm10 scaled \magstep1
\font\mbs=msbm7 scaled \magstep1
\font\mbss=msbm5 scaled \magstep1
\newfam\mbff
\textfont\mbff=\mbn
\scriptfont\mbff=\mbs
\scriptscriptfont\mbff=\mbss   

\newcommand{\Di}      {\mathbb{D}}

\newcommand{\N}       { \mathbb{N}}
  
\newcommand\Co           {{\mathbb C}}

\newtheorem{Th}{Theorem}[section]
\newtheorem{Lm}[Th]{Lemma}
\newtheorem{C}[Th]{Corollary}
\newtheorem{D}[Th]{Definition}

\newtheorem{Prop}[Th]{Proposition}
\newtheorem{R}[Th]{Remark}

\newtheorem{E}[Th]{Example}

\newtheorem*{Lemma A}{Lemma A}
\newtheorem*{Lemma B}{Lemma B}
\newtheorem*{Lemma C}{Lemma C}

\newtheorem*{Th A}{Theorem A}
\newtheorem*{Th B}{Theorem B}
\begin{document}

\title[On homomorphisms of Douglas algebras]{On homomorphisms of Douglas algebras}
\author{Alexander Brudnyi}
\address{Department of Mathematics and Statistics\newline
\hspace*{1em} University of Calgary\newline
\hspace*{1em} Calgary, Alberta, Canada\newline
\hspace*{1em} T2N 1N4}
\email{abrudnyi@ucalgary.ca}


\thanks{Research is supported in part by NSERC}

\begin{abstract} 
The paper describes homomorphisms between Douglas algebras and some semisimple Banach algebras. The main tool is a result on the structure of the space
$C(Z,\mathfrak M)$ of continuous mappings from a connected first-countable $T_1$ space $Z$ to the maximal ideal space $\mathfrak M$ of the algebra $H^\infty$ of bounded holomorphic functions on the unit disk $\Di\subset\Co$. In particular, it is shown that the space of continuous mappings from  $Z$ to $\Di$ is dense in the topology of pointwise convergence in $C(Z,\mathfrak M)$ and
the homotopy groups of $\mathfrak M$ are trivial.
 \end{abstract}

\date{}

\maketitle

\sect{Formulation of Main Results}
\subsect{} 
Recall that for a commutative unital complex Banach algebra $A$,  the maximal ideal space $\mathfrak M(A)\subset A^\ast$  
 is the set of nonzero homomorphisms $A \!\rightarrow\! \Co$ endowed with the Gelfand topology, the weak-$\ast$ topology of  $A^\ast$. It is a compact Hausdorff space contained in the unit sphere of $A^\ast$. The {\em Gelfand transform} defined by $\hat{a}(\varphi):=\varphi(a)$ for $a\in A$ and $\varphi \in \mathfrak M(A)$ is a nonincreasing-norm morphism from $A$ into the Banach algebra $C(\mathfrak M (A))$ of complex-valued continuous functions on $\mathfrak M(A)$. For $A$ a Banach complex function algebra on a topological space $X$, there is a continuous mapping $\iota:X\hookrightarrow \mathfrak M(A)$ taking $x\in X$ to the evaluation homomorphism $f\mapsto f(x)$, $f\in A$, injective if $A$ separates the points of $X$. 
 
 Let $H^\infty$ be the Banach algebra of bounded holomorphic functions on the open unit disk $\Di\subset\Co$ equipped with pointwise multiplication and supremum norm and let $\mathfrak M$ be its maximal ideal space. Then the  Gelfand transform  $\,\hat{\,}\,: H^\infty\rightarrow C(\mathfrak M)$ is an isometry and $\iota:\Di\hookrightarrow \mathfrak M$ is an embedding with dense image by the celebrated Carleson corona theorem \cite{C}. In the sequel, we identify $\Di$ with its image under $\iota$.

Let 
\begin{equation}\label{e1.1}
\rho(z,w):=\left|\frac{z-w}{1-\bar w z}\right|,\qquad z,w\in\mathbb D,
\end{equation}
be the pseudohyperbolic metric on $\Di$.

For $x,y\in \mathfrak M$ the formula
\begin{equation}\label{e1.2}
\rho(x,y):=\sup\{|\hat f(y)|\, :\, f\in H^\infty,\, \hat f(x)=0,\, \|f\|_{H^\infty}\le 1\}
\end{equation}
gives an extension of  $\rho$  to $\mathfrak M$.
The {\em Gleason part} of $x\in \mathfrak M$ is then defined by $\pi(x):=\{y\in \mathfrak M\, :\, \rho(x,y)<1\}$. For $x,y\in \mathfrak M$ we have
$\pi(x)=\pi(y)$ or $\pi(x)\cap\pi(y)=\emptyset$. Hoffman's classification of Gleason parts \cite{H} shows that there are only two cases: either $\pi(x)=\{x\}$ or $\pi(x)$ is an analytic disk. The former case means that there is an {\em analytic parametrization}
 of $\pi(x)$, i.e., a continuous one-to-one and onto mapping $L:\Di\to\pi(x)$ such that $\hat f\circ L\in H^\infty$ for every $f\in H^\infty$.
Moreover, any analytic disk is contained in a Gleason part and any maximal (i.e., not contained in any other) analytic disk is a Gleason part. By $\mathfrak M_a$ and $\mathfrak M_s$ we denote the sets of all non-trivial (analytic disks) and trivial (one-pointed) Gleason parts, respectively. It is known that $\mathfrak M_a\subset \mathfrak M$ is open. Hoffman proved that $\pi(x)\subset \mathfrak M_a$ if and only if $x$ belongs to the closure of an interpolating sequence for $H^\infty$.  

Let $L^\infty$ be the Banach algebra  of essentially bounded Lebesgue measurable functions on the unit circle $\mathbb S$ (with pointwise operations and the supremum norm). Via identification with boundary values, $H^\infty$ is a uniformly closed subalgebra of $L^\infty$.  Due to the Chang-Marshall theorem, see, e.g., \cite[Ch.\,IX,\,\S 3]{Ga}, any uniformly 
closed subalgebra $\mathscr D$ between $H^\infty$ and $L^\infty$ is a {\em Douglas algebra} generated by $H^\infty$ and a family $\mathscr U_\mathscr D\subset\overline{H^\infty}$ of functions conjugate to some inner functions of $H^\infty$ (written $\mathscr D=[H^\infty,\mathscr U_\mathscr D]$). The maximal ideal space $\mathfrak M(\mathscr D)$ of $\mathscr D$ is a closed subset of 
$\mathfrak M$ of the form (see, e.g., \cite[Ch.\,IX,\,Thm.\,1.3]{Ga}):
\begin{equation}\label{eq1.3}
\mathfrak M(\mathscr D)=\bigcap_{\bar u\in \mathscr U_\mathscr D}\{x\in \mathfrak M\, :\, |
\hat u(x)|=1\}.
\end{equation}
In particular, by the maximum modulus principle for holomorphic functions, $\mathfrak M(\mathscr D)\cap\mathfrak M_a$ is either empty or  the union of some Gleason parts. 

Our first result describes the structure of the set ${\rm Hom}(\mathscr D, A)$ of unital (i.e., sending units to units) homomorphisms from $\mathscr D=[H^\infty,\mathscr U_\mathscr D]$ to a certain class of Banach complex function algebra $A$. For its formulation, we require
\begin{D}\label{def1.1}
 A topological space $Z$ is  of class $\mathscr C$  if each pair of points of $Z$ belongs to the image of a continuous mapping from a connected first-countable $T_1$ space to $Z$. 
 \end{D}
\noindent For instance, $\mathscr C$ contains direct products of connected first-countable $T_1$ spaces and path-connected spaces.

Recall that a topological space is $T_1$  if for every pair of distinct points, each has a neighbourhood not containing the other point.  A topological space is {\em first-countable} if each point has a countable neighbourhood basis. (E.g., a metric space is a first-countable $T_1$ space.) 

Suppose $Z$ is a compact topological space whose
connected components are of class $\mathscr C$ and the cardinality of the set of connected components is less than  $2^{\mathfrak c}$, where $\frak c$ is the cardinality of the continuum. Let $A$ be a Banach complex function algebra on $Z$. We denote by $\mathbf 1_Z$ the unit of $A$ (i.e., the constant function of value $1$ on $Z$). For a nonzero idempotent $p\in A$, we set $A_p:=\{pg\, :\, g\in A\}$. Then $A_p$
is a closed subalgebra of $A$ with unit $p$.  
\begin{Th}\label{te1.2}
For each $\Phi\in {\rm Hom}(\mathscr D, A)$ there exist
pairwise distinct idempotents and Gleason parts $p_i\in A$ and $\pi_i\subset \mathfrak M(\mathscr D)$, $1\le i\le k$, $k\in\N$, such that 
\begin{equation}\label{eq1.4}
\Phi=p_1\cdot\Phi+\cdots +p_k\cdot\Phi 
\end{equation}
and each homomorphism $p_i\cdot\Phi\in {\rm Hom}(\mathscr D, A_{p_i})$ is a compact linear operator of one of the  forms:
\begin{itemize}
\item[(a)]
If $\pi_i=\{x_i\}\subset\mathfrak M(\mathscr D)\cap\mathfrak M_s$, then
\begin{equation}\label{eq1.5}
(p_i\cdot \Phi)(f)=\hat f(x_i)\cdot  p_i,\quad f\in \mathscr D; 
\end{equation}
\item[(b)]
If $\pi_i\subset \mathfrak M(\mathscr D)\cap\mathfrak M_a$ with an analytic parametrization $L_i$, then
there exists a function $g_i\in A_{p_i}$ with a compact range in $\Di$ such that
\begin{equation}\label{eq1.6}
(p_i\cdot \Phi)(f)=(\hat f\circ L_i\circ g_i)\cdot p_i,\quad f\in \mathscr D.
\end{equation}
\end{itemize}
\end{Th}
\begin{R}\label{rem1.3}
{\rm Under the hypothesis of the theorem, 
\[
{\rm Ker}\,\Phi=\bigcap_{i=1}^k {\rm Ker}\, (p_i\cdot\Phi),
\] 
where
\begin{itemize}
\item[(a$_1$)] If $\pi_i=\{x_i\}$, then  
\[
{\rm Ker}\, (p_i\cdot\Phi)={\rm Ker}\,x_i,
\]
a maximal ideal of $\mathscr D$;\smallskip
\item[(b$_1$)]
If $\pi_i\subset \mathfrak M(\mathscr D)\cap\mathfrak M_a$  and the range of $g_i$ restricted to ${\rm supp}\, p_i$, the support of $p_i$, is a finite set, say, $\{z_{i_1},\dots,z_{i_l}\} \subset\Di$, then 
\[
{\rm Ker}\, (p_i\cdot\Phi)=\bigcap_{j=1}^l{\rm Ker}\,L_i(z_{i_j});\smallskip
\]
\item[(b$_2$)]
 If $\pi_i\subset \mathfrak M(\mathscr D)\cap\mathfrak M_a$  and the range of $g_i$ restricted to ${\rm supp}\, p_i$ is infinite, then 
 \[
 {\rm Ker}\, (p_i\cdot\Phi)=\{f\in\mathscr D\, :\, \hat f|_{\pi_i}=0\}.
 \]
 \end{itemize}
Indeed, in the final case, since $|\hat u|=1$ on $\pi_i$ for $\bar u\in \mathscr U_\mathscr D$, the function $L_i^*\bar u:=\bar u\circ L_i$ is constant and hence $L_i^*(\mathscr D)\subset H^\infty$. Also, $g_i({\rm supp}\,p_i)$ is an infinite compact subset of $\Di$. Thus, if  
$\hat f\circ L_i=0$ on $g_i({\rm supp}\,p_i)$ for some $f\in\mathscr D$, then $\hat f\circ L_i=0$ on $\Di$, i.e., $\hat{f}|_{\pi_i}=0$.
}
\end{R}
\begin{E}\label{ex1.4}
{\rm  Let $A$ be a separable commutative semisimple unital complex Banach algebra. Then
the closed unit ball of the dual space $A^\ast$ equipped with the weak-$\ast$ topology is a metrizable compact space; hence, its cardinality is $\le\mathfrak c\, (<2^{\frak c})$. In particular, $\mathfrak M(A)\in\mathscr C$. Since $A$ is semisimple, the Gelfand transform $\hat{\, }: A\to C(\mathfrak M(A))$ is a Banach algebra monomorphism. Thus, $A$ can be identified via $\hat{\, }$ with a Banach complex function algebra on $\mathfrak M(A)$ and so
{\em  the structure of the set
${\rm Hom}(\mathscr D,  A)$, $\mathscr D=[H^\infty,\mathscr U_\mathscr D]$, is described by Theorem \ref{te1.2}}.}
\end{E}
\subsect{}
Theorem \ref{te1.2} is based on some results about
continuous mappings to the maximal ideal space of $H^\infty$ formulated in this section.

In what follows, we work in the category of pointed topological spaces with pointed mappings (i.e., basepoint preserving continuous mappings) as morphisms.

\begin{Th}\label{te1.5}
Let $F: (Z,z)\rightarrow (\mathfrak M, x)$ be a pointed mapping from a topological space $Z$ of class $\mathscr C$ to $\mathfrak M$. Then $F(Z)\subset \pi(x)$.
\end{Th}
 A topological space $X$ is said to be {\em compactly generated} if it satisfies the following condition:
a subspace $S$ is closed in $X$ if and only if $S \cap K$ is closed in $K$ for all compact subspaces $K\subseteq X$. For instance, direct products of first-countable spaces and locally compact spaces are compactly generated, see, e.g., \cite{St}.\smallskip

 As a corollary of Theorem \ref{te1.5} we obtain:
 
\begin{Th}\label{te1.6}
Let $Z$ be a compactly generated space of class $\mathscr C$.
 \begin{itemize}
 \item[(i)]
 The space of continuous mappings from  $Z$ to $\Di$ is dense in the topology of pointwise convergence in $C(Z,\mathfrak M)$.\smallskip
 \item[(ii)] 
 Every pointed mapping in $C(Z,\mathfrak M)$ is pointed null homotopic. In particular,
 all
 homotopy groups $\pi_n(\mathfrak M,x)$, $n\in\N$, are trivial.
 \end{itemize}
\end{Th}

\begin{R}\label{rem1.4}
{\rm (1)
The special case of part (i) of the theorem for
the space of analytic mappings from a connected analytic space to $\mathfrak M$ was initially proved  by Hoffman  \cite[Thm.\,5.6]{H}.\smallskip

\noindent (2) Theorem \ref{te1.5} is motivated by \cite[Thm.\,1.2\,(d)]{Br1} stating that if  $F: (Y,y)\rightarrow\mathfrak (\mathfrak M_a,x)$ is a pointed mapping of a connected and locally connected space $Y$, then  $F(Y)\subset \pi(x)$. It is not clear whether the same holds with $\mathfrak M$ in place of $\mathfrak M_a$.\smallskip

\noindent (3) Theorem \ref{te1.5} implies, in particular,
that Gleason parts are path components of $\mathfrak M$, the result originally obtained in \cite[Cor.\,12]{AG}.

\noindent (Recall that  a {\em path component} of a topological space $X$ is an equivalence class of $X$ under the equivalence relation:  $x,y\in X$ are equivalent  if there is a path from $x$ to $y$.)
}
\end{R}

\noindent {Acknowledgement.} I thank  R. Moritini for calling my attention to reference \cite{AG}.

\sect{Proofs of Theorems \ref{te1.5} and  \ref{te1.6}}
We derive Theorem \ref{te1.5} from the following more general result.

Let $Y$ be a topological space satisfying the properties
\begin{itemize}
\item[(a)]
There exists a family of mutually disjoint subsets $Y_i\subset Y$, $i\in I$, such that 
\[
Y=\bigsqcup_{i\in I} Y_i.
\]
\item[(b)]
If a sequence $\{y_n\}_{\in\N}\subset Y$ converges to some $y\in Y_i$, then $y_n\in Y_i$ for all sufficiently large $n$.
\end{itemize}
\begin{Prop}\label{prop2.1}
Let $F: (Z,z)\rightarrow (Y, y)$ be a pointed mapping from a topological space $Z$ of class $\mathscr C$ to $Y$. If $y\in Y_i$, then $F(Z)\subset Y_i$.
\end{Prop}
\begin{proof}
By the definition of class $\mathscr C$, see Definition \ref{def1.1}, it suffices to prove the result for a connected first-countable $T_1$ space $Z$.

Now, by (a) $Z$ is the disjoint union of subsets $F^{-1}(Y_i)$, $i\in I$.   We prove that each  $F^{-1}(Y_i)\ne\emptyset$ is an open subset of $Y$. Indeed, for otherwise there is a boundary point $y\in F^{-1}(Y_i)$ of $F^{-1}(Y_i)$. Then, since $Z$ is a first-countable $T_1$-space, there is a  sequence $\{y_n\}_{n=1}^\infty\subset Z\setminus F^{-1}(Y_i)$ converging to $y$. Hence, the sequence $\{F(y_n)\}_{n=1}^\infty\subset Y\setminus Y_i$ converges to $F(y)\in Y_i$. But by (b), $F(y_n)\in Y_i$ for all sufficiently large $n$, a contradiction. 
The latter implies that $Z$ is the disjoint union of open subsets $F^{-1}(Y_i)$.
Then, since $Z$ is connected, $F^{-1}(Y_i)=\emptyset$ for all $Y_i$ which do not contain $y$. Hence, $F$ maps $(Z,z)$ into $(Y_i,y)$.
\end{proof}
\begin{proof}[Proof of Theorem \ref{te1.5}]
We set $Y:=\mathfrak M$ and take Gleason parts of $\mathfrak M$ as sets $Y_i$. Then property (b) of $Y$ follows from \cite[Prop.\,1]{D}. Now, the required result is a consequence of Proposition \ref{prop2.1}.
\end{proof}
\begin{E}\label{ex2.2}
{\rm Let $H^\infty(\Omega)$ be the Banach algebra of bounded holomorphic functions on a nonempty open subset $\Omega$ of $\Co^N$. We set $Y:=\mathfrak M(H^\infty(\Omega))$, the maximal ideal space of 
$H^\infty(\Omega)$, and take Gleason parts of $\mathfrak M(H^\infty(\Omega))$ as sets $Y_i$. Then property (b) of $Y$ follows from \cite[Cor.\,9]{AG}.
}
\end{E}
 
 To prove Theorem \ref{te1.6} we require the following 
 result.

Let $L:\Di\rightarrow\pi(x)\subset\mathfrak M_a$ be an analytic parametrization of $\pi(x)$.
\begin{Prop}\label{prop2.3}
A subset $K\subset \pi(x)$  is compact in the topology of $\mathfrak M$ if and only if $L^{-1}(K)$ is a compact subset of $\Di$.
\end{Prop}
\begin{proof}
According to Hoffman, see \cite[page\, 104]{H}, the base of the topology of $\mathfrak M_a$ consists of compact subsets homeomorphic to $\bar{\Di}\times \beta\N$ (here $\bar{\Di}\subset\Co$ is the closure of $\Di$ and $\beta\N$ is the Stone-\v{C}ech compactification on $\N$). Let 
$(S_i)_{i=1}^s\subset\mathfrak M_a$ be a finite cover of $K$ by such sets, $\phi_i:S_i\to \bar{\Di}\times\beta\N$, $1\le i\le s$, be the corresponding homeomorphisms, and $p:\bar{\Di}\times\beta\N\to\beta\N$ be the projection onto the second coordinate. 
Consider the compact sets $K_i=K\cap S_i\subset\pi(x)$, $1\le i\le s$.
Then $p\circ\phi_i $ maps $K_i$ onto a compact subset $K_i'\subset\beta\N$.
Since every infinite compact subset of $\beta\N$ contains a  homeomorphic copy of $\beta\N$, its cardinality  is $2^{\mathfrak c}$. On the other hand, the cardinality of $K_i'$ does not exceed $\mathfrak c$, the cardinality of $\pi(x)$. Hence, $K_i'$ is a finite set, say, $K_i'=\{t_{i1},\dots, t_{il_i}\}$.
 In particular, $K_i=\sqcup_{j=1}^{l_i} K_{ij}$, where $K_{ij}
\subset\pi(x)$ is compact  such that  $\phi_i(K_{ij})\subset\bar{\Di}\times\{t_{ij}\}$.
 Further, if for some $i,j$  the closed set $L^{-1}(K_{ij})\subset\Di$ is not compact, there is a sequence $\{z_n\}\subset L^{-1}(K_{ij})$ tending to $\mathbb S$ such that the sequence $\{(\phi_i\circ L)(z_n)\}\subset \bar{\Di}\times \{t_{ij}\}$ converges; hence,
 the sequence $\{L(z_n)\}\subset K_{ij}$ converges to some $x\in K_{ij}$. Due to \cite[Cor.\,13]{AG} this implies that $\lim_{n\to\infty}\rho(L(z_n),x)=0$ (see \eqref{e1.2} for the definition of $\rho$). On the other hand,  $L: (\Di,\rho)\rightarrow (\pi(x),\rho)$ is an isometric isomorphism (see \cite{H}, page 105, equation (6.12)). Hence, $\lim_{n\to\infty} z_n =L^{-1}(x)\in\Di$, a contradiction which proves the result.
\end{proof}
As a corollary of the proposition we obtain:
\begin{C}\label{cor2.4}
The mapping $L^{-1}:\pi(x)\rightarrow\Di$ is continuous on compact subsets.
\end{C}
\begin{proof}
Let $K\subset\pi(x)$ be a compact subset. Then due to Proposition \ref{prop2.3}, $L^{-1}(K)=:K'$ is a compact subset of $\Di$. Since $L|_{K'}:K'\rightarrow K$ is a continuous bijection of compact spaces, it is a homeomorphism, i.e., $L^{-1}|_{K}:K\rightarrow K'$ is continuous.
\end{proof}
\begin{proof}[Proof of Theorem \ref{te1.6}]
Let $F: (Z,z)\rightarrow (\mathfrak M, x)$ be a pointed mapping from a compactly generated space $Z$ of class $\mathscr C$ to $\mathfrak M$. Due to Theorem \ref{te1.5}, $F(Z)\subset\pi(x)$. If $\pi(x)$ is an analytic disk, we denote by $L_x:\Di\rightarrow\pi(x)$  an analytic parametrization  sending $0$ to $x$. 
\begin{Lm}\label{lem2.5}
There is a continuous mapping $\widetilde F: (Z,z)\rightarrow (\Di,0)$ such that $F=L_x\circ\widetilde F$.
\end{Lm}
\begin{proof}
We set
\begin{equation}\label{eq6.1}
\widetilde F=L_x^{-1}\circ F.
\end{equation}
Since $Z$ is compactly generated, $\widetilde F$ is continuous if and only if $\widetilde F|_K$ is continuous for each compact subset $K \subseteq Z$, see, e.g., \cite{St}.
For such $K$ continuity of $F$ implies that $F(K)\subset\pi(x)$ is compact.
Then due to  Corollary \ref{cor2.4}, $L_x^{-1}$ is continuous on $F(K)$. Thus, $\widetilde F|_K=L_x^{-1}|_{F(K)}\circ F|_K $ is continuous, as required. 
\end{proof}
Now, we are ready to prove the theorem.

(i) If $\pi(x)=\{x\}$,  by the corona theorem $F$ is a pointwise limit of a net of constant mappings with values in $\Di$. For otherwise, let $(L_\alpha)_{\alpha\in A}$ be a net of M\"{o}bius transformations of $\Di$ converging pointwise to $L_x$ (see, e.g., \cite[Ch.\,X]{Ga} for its existence). Then due to Lemma \ref{lem2.5} the net $(L_\alpha\circ\widetilde F)_{\alpha\in A}$ of continuous mappings from $Z$ to $\Di$ converges pointwise to $F$.

 (ii) If $\pi(x)=\{x\}$,  the mapping $F$ is constant and so is trivially pointed null  homotopic. In turn, if $\pi(x)$ is an analytic disk, then since the continuous mapping  $\widetilde F: (Z,z)\rightarrow (\Di,0)$ is pointed null homotopic, $F: (Z,z)\rightarrow (\mathfrak M, x)$ is pointed null homotopic as well.
 
 The proof of the theorem is complete.
\end{proof}

\sect{Proof of Theorem \ref{te1.2}}
Each $\Phi\in {\rm Hom}(\mathscr D,A)$, $\mathscr D=[H^\infty,\mathscr U_\mathscr D]$, is a bounded linear operator, see, e.g., \cite[\S24B, Theorem]{L}. Hence, the transpose $\Phi^*: A^*\to \mathscr D^*$ is well defined and maps the maximal ideal space $\mathfrak M(A)$ of $A$ continuously into $\mathfrak M(\mathscr D)$.
Further, every evaluation homomorphism $\delta_z$ at a point $z\in Z$ is an element of $\mathfrak M(A)$. Thus, since $A\subset C(Z)$, the mapping $\Delta: Z\to\mathfrak M(A)$, $z\mapsto\delta_z$,  is continuous. The composition $\Phi^*\circ\Delta$ is a continuous mapping from $Z$ to $\mathfrak M(\mathscr D)\subset\mathfrak M$. 
\begin{Lm}\label{lem3.1}
The compact set $(\Phi^*\circ\Delta)(Z)$ meets finitely many   Gleason parts of $\mathfrak M(\mathscr D)$.
\end{Lm}
\begin{proof}
By the hypotheses,  each connected component of $Z$ is of class $\mathscr C$. Hence, by Theorem \ref{te1.5},  its image under $\Phi^*\circ\Delta$ lies in a Gleason part of $\mathfrak M(\mathscr D)$. Since by our assumption the cardinality of the set of connected components of $Z$ is less than $2^{\mathfrak c}$, the set of  Gleason parts which meet $(\Phi^*\circ\Delta)(Z)$ is of cardinality $<2^{\mathfrak c}$ as well. 
Assume, on the contrary, that $(\Phi^*\circ\Delta)(Z)$ meets infinitely many  Gleason parts. Then due
to \cite[Prop.\,2]{D} it has cardinality $\ge 2^{\mathfrak c}$, a contradiction which proves the lemma. 
\end{proof}
The lemma implies that there exist distinct Gleason parts $\pi_i$, $1\le i\le k$, such that $\Phi^*\circ\Delta$ maps each connected component of $Z$ to one of $\pi_i$. We set 
\[
K_i:=(\Phi^*\circ\Delta)(Z)\cap\pi_i,\quad 1\le i\le k.
\]
\begin{Lm}\label{lem3.2}
Each $K_i\subset\pi_i$ is compact.
\end{Lm}
 \begin{proof}
 Clearly, it suffices to prove the result for  $\pi_i$ an analytic disk. Let $L_i:\Di\to\pi_i$ be an analytic parametrization. Suppose, on the contrary, that $K_i$ is not compact. Then there is a sequence $\{z_n\}_{n\in\N}\subset\Di$ tending to $\mathbb S$ such that
all $L_i(z_n)\in K_i$. Let $x\in K$ be a limit point of $\{L_i(z_n)\}_{n\in\N}$. Due to \cite[Thm.\,3]{AG}, if $x\not\in\pi$, then the cardinality of the closure of $\{L_i(z_n)\}_{n\in\N}$ is $2^{\mathfrak c}$. However, $K$ belongs to finitely many Gleason parts and so has cardinality at most $\mathfrak c$. This shows that all limit points of $\{L_i(z_n)\}_{n\in\N}$ belong to $\pi_i$, i.e., the closure of the set is a compact subset of $\pi_i$. Then due to Proposition \ref{prop2.3}, $\{z_n\}_{n\in\N}$ is a relatively compact subset of $\Di$, a contradiction which proves the lemma.
\end{proof}
As a direct corollary of the lemma we obtain that all
$Y_i:=(\Phi^*)^{-1}(K_i)$ are clopen subsets of $\mathfrak M(A)$ and all $Z_i:=(\Phi^*\circ\Delta)^{-1}(K_i)$ are clopen subsets of $Z$ formed by disjoint unions of some connected components. Hence, by the Shilov idempotent theorem, see, e.g., \cite[Ch.\,III, Cor.\,6.5]{G}, there exist idempotents $p_i\in A$, $1\le i\le k$, such that ${\rm supp}\, p_i=Z_i$ and 
$\Phi=\sum_{i=1}^k\, p_i\cdot\Phi$. Here each 
$p_i\cdot\Phi\in {\rm Hom}(\mathscr D,A_{p_i})$ and the maximal ideal space of $A_{p_i}$ is  $Y_i$.
The transpose of $p_i\cdot\Phi$ restricted to $Z_i$ coincides with the mapping $(\Phi^*\circ\Delta)|_{Z_i}$ and,  in particular, its image is $K_i\subset\pi_i$.
To complete the proof of the theorem we must show that each $p_i\cdot\Phi$ is given by \eqref{eq1.5} or \eqref{eq1.6} and that $p_i\cdot\Phi$ is a compact operator.

To this end, we consider two cases:\smallskip

1. $\pi_i:=\{x_i\}$ is a trivial Gleason part.

Then, for every $f\in \mathscr D$ and all $z\in Z_i$,
\begin{equation}\label{eq3.1}
((p_i\cdot\Phi)(f))(z)=\delta_z(\Phi(f))=\hat f((\Phi^*\circ\Delta)(z))=\hat f(x_i).
\end{equation}
 Thus, $(p_i\cdot \Phi)(f)=\hat f(x_i)\cdot p_i$ for all $f\in \mathscr D$ as in \eqref{eq1.5}.\smallskip
 
 2. $\pi_i$ is an analytic disk.  
 
 Let $L_i:\Di\to\pi_i$ be an analytic parametrization. Since the image of $(\Phi^*\circ\Delta)|_{Z_i}$ is a compact subset of $\pi_i$, there is a function $g_i\in C(Z_i)$ with image in $\Di$ such that $L_i\circ g_i= (\Phi^*\circ\Delta)|_{Z_i}$ (cf. Lemma \ref{lem2.5}). Thus,
 \begin{equation}\label{eq3.2}
 ((p_i\cdot\Phi)(f))(z)=\delta_z(\Phi(f))=\hat f((\Phi^*\circ\Delta)(z))=(\hat f\circ L_i\circ g_i)(z),\quad f\in \mathscr D,\ z\in Z_i.
 \end{equation}
 To show that $p_i\cdot\Phi$ is as in \eqref{eq1.6}, we must prove the following
 \begin{Lm} \label{lem7.1}
Function $g_i$ belongs to $A|_{Z_i}$.
 \end{Lm}
 \begin{proof}
Below, we naturally identify $A_{p_i}$ with $A|_{Z_i}$.
 
Since $g(Z_i)\subset\Di$ is compact,  by Hoffman's theorem  \cite[Ch.\,X,\,Thm.\,1.7]{Ga}  there is
 an interpolating Blaschke product $B$ such that the function $\hat B\circ L_i\in H^\infty\subset \mathscr D$ is one-to-one in a neighbourhood of $g(Z_i)$. 
 Then there is a function $h_i$, holomorphic in a neighbourhood of the compact set $S_i:=(\hat B \circ L_i)(g(Z_i))$ such that $h_i|_{S_i}$ is inverse of $(\hat B \circ L_i)|_{g(Z_i)}$.  Due to \eqref{eq3.2},  $\hat B \circ L_i\circ g_i\in A|_{Z_i}$, and its spectrum is $S_i$.
 Since  $h_i$ is holomorphic in a neighbourhood of $S_i$, the functional calculus for $\hat B \circ L_i\circ g_i$, see, e.g., \cite[Ch.\,10.26]{R}, implies that
 $g_i=h_i\circ \hat B \circ L_i\circ g_i$ lies in $A|_{Z_i}$ as well. 
 \end{proof}
 From the lemma we obtain that $p_i\cdot\Phi$ given by \eqref{eq3.2} is as in \eqref{eq1.6}.\smallskip
 
Finally, let us show that
each $p_i\cdot\Phi\in {\rm Hom}(\mathscr D,A_{p_i})$ is a compact operator.

It is clear for $p_i\cdot\Phi$ satisfying \eqref{eq3.1}. In turn, if $p_i\cdot\Phi$ satisfies \eqref{eq3.2}, then the spectrum  of $g_i\in A|_{Z_i}$ is $g(Z_i)\subset \Di_r:=\{z\in\Co\, :\, |z|<r\}$ for some $r\in (0,1)$.
Also, recall that $L_i^*(\mathscr D)\subset H^\infty$, see Remark \ref{rem1.3}. 
Thus, $p_i\circ\Phi$ is the composition of the homomorphism $L_i^*:\mathscr D\to H^\infty$, $f\mapsto \hat f\circ L_i$, the restriction homomorphism $R: H^\infty\to H^\infty(\Di_r)$, $f\mapsto f|_{\Di_r}$, and the homomorphism $H^\infty(\Di_r)\to A_{p_i}$ determined by the functional calculus of $g_i\in A|_{Z_i}$, see, e.g., \cite[Thm.\,10.27]{R}. 
Since the operator $R$ is compact, $p_i\cdot\Phi$ is a compact operator as well.

The proof of the theorem is complete.

\end{document}